\documentclass[reqno]{amsart}

\usepackage{amssymb,amsfonts,amstext,amsthm,hyperref,cleveref,xcolor}

\usepackage{graphics}
\usepackage{graphicx}
\usepackage{epstopdf}
\usepackage{indentfirst}
\usepackage{float}
\usepackage[numbers,sort&compress]{natbib}
\bibpunct[, ]{[}{]}{,}{n}{,}{,}
\makeatletter
\def\NAT@def@citea{\def\@citea{\NAT@separator}}
\makeatother
\theoremstyle{plain}
\newtheorem{theorem}{Theorem}[section]
\newtheorem{lemma}[theorem]{Lemma}
\newtheorem{corollary}[theorem]{Corollary}
\newtheorem{proposition}[theorem]{Proposition}
\crefrangeformat{equation}{#3(#1)#4--#5(#2)#6}
\crefname{enumi}{\unskip}{\unskip}

\theoremstyle{definition}

\newtheorem{remark}[theorem]{Remark}

\begin{document}

\title[Counting Elements on Full Matrix groups]{Counting Elements on Full Matrix Groups over finite Field with prescribed eigenvalues}

\author{Ivan Gargate}
\address{UTFPR, Campus Pato Branco, Rua Via do Conhecimento km 01, 85503-390 Pato Branco, PR, Brazil}
\email{ivangargate@utfpr.edu.br}

\author{Michael Gargate}
\address{UTFPR, Campus Pato Branco, Rua Via do Conhecimento km 01, 85503-390 Pato Branco, PR, Brazil}
\email{michaelgargate@utfpr.edu.br}

\begin{abstract}
In the present article we shown a formula to compute the number of all matrices over the finite field $F$ whit prescribed eigenvalues. Using this formula we obtain one inequality for the number of $(k+1)$-potent elements over finite rings. 
\end{abstract}



\maketitle

\section{Introduction}\label{intro}

Let $F$ be a finite field with $|F|=q$ and denote by $M_n(F)$ the  $n\times n$ matrix algebra with entries over $F$. Compute the number of special elements over the group of full matrix or the group of triangular matrix is realized for various authors, for instance, in triangular matrix groups we have the works of Slowik \cite{Slowik} and Hou \cite{Hou} that compute the number of involutions and idempotents, respectively,  Slowik in \cite{Slowik2} compute the number of matrices of finite order. For the incidence algebras of a finite poset, the authors in \cite{Gargate4} compute the number of involutions depending of their respective poset. Recently, the authors in \cite{Gargate1}  and \cite{Gargate3} compute the number of $(k+1)$-potent matrices and the matrices such that are solutions of one quadratic equation. Also in \cite{Gargate2} the authors compute the number of coninvolution matrices over the ring of Gaussian Integers module $p$ and the Quaternion Integers module $p$. For full matrix groups Cheraghpour and Ghosseiri \cite{Ghosseiri} compute the number of idempotent and nilpotent elements.

Various special and important matrices can be studied by  know  their eigenvalues, for instance:
\begin{itemize}
    \item [1.)]If $\alpha=0$ and
 $\theta=1$ are eigenvalues of a matrix $A$ then $A$ is an idempotent matrix.
 \item[2.)] If $\alpha=1$ and $ \theta=-1$  we have an involution.
 \item[3.)] If $\alpha=1,\beta=\omega$, $\theta=\omega^2$ with $\omega$ a cube root of unity then we have a matrix of order $3$ (i.e $A^3=I$).
 \item[4.)] If $\alpha=0,\beta=1,\theta=-1$ then we have a 3-potent matrix. ($A^3=A$).
  \end{itemize}
 
The propose of our article is compute the number of matrices with prescribed eigenvalues, for example, the number of involutions, the number of $k$-potent matrices, i.e. $A^k=A$, and the number of matrices of finite order $k$, i.e. $A^k=I$ with $I$ the identity in $R$ and other special matrices.

The main results of our article is stated as follows:

\begin{theorem}\label{th1} Let $\alpha_1,\alpha_2,\cdots,\alpha_k \in F$ be all different numbers, $R=M_n(F)$ the $n\times n$ matrix group  and define the subsets 
$$M(\alpha_1,\alpha_2,\cdots\!,\alpha_k)\!=\!\{A\in R,  if \ \gamma \ is \ an \ eigenvalue \ of \ A \ then \ \gamma\in \!\{  \alpha_1,\alpha_2, \cdots,\alpha_k\} \},$$
 and 
 $$E(\alpha_1,\alpha_2,\cdots, \alpha_k)=\{A \in R,  \alpha_1,\alpha_2,\cdots \ldots, \alpha_k \ are  \ eigenvalues \ of \ A\}.$$
Then
 $$|M(\alpha_1,\alpha_2,\cdots,\alpha_k)|=\sum_{\substack{n_1+n_2+\cdots+n_k=n\\ n_i\geq 0, i=1,2,\cdots,k}}\frac{U}{U_{n_1}U_{n_2}\cdots U_{n_k}}.$$
 And 
 $$|E(\alpha_1,\alpha_2,\cdots \alpha_k)|=\sum_{\substack{n_1+n_2+\cdots n_k=n\\ n_i \geq 1, i=1,2,\cdots,k}}\frac{U}{U_{n_1}U_{n_2}\cdots  U_{n_k}},$$
 where $U=|U_n(M_n(F))|$ and $U_{n_j}=|U(M_{n_j}(F))|$ are the number of units of $M_n(F)$ and $M_{j}(F)$ for $j=1,2,\cdots,k$, respectively, and by convention, $U_0=1$.
  \end{theorem}
  
 In the section 3 we use the Theorem \ref{th1}  for proof the following upper bound  that generalize the MacHale's results (see \cite{Machale}) for $(k+1)$-potent elements over finite rings  :
 
 \begin{theorem}\label{th2}
Let $p$ be a prime and $R$ be a finite ring such that $|R|$ is some power of $p$. Then
$$\mathcal{I}_{k+1}(R)|\leq \displaystyle\frac{k+1}{p} |R|^{\displaystyle\frac{2k}{k+1}},$$
where $\mathcal{I}_{k+1}$ denote the set of $(k+1)$-potent elements in $R$, $k\geq 1$.
\end{theorem}
 
 And in the general case
 
 \begin{theorem}\label{th3}
 If $R$ is a finite ring and $p_i$, $1\leq i \leq n$, are the different primes dividing $|R|$, then
$$|\mathcal{I}_{k+1}(R)|\leq \displaystyle \frac{(k+1)^n}{\prod^n_{i=1}p_i}|R|^{\displaystyle \frac{2k}{k+1}},$$
where $\mathcal{I}_{k+1}(R)$ denote the set of $(k+1)$-potent elements in $R$.
\end{theorem}

\section{Counting Matrices with prescribed eigenvalues}
 
 Remember that two matrices $A$ and $B$ are similar if there exists an invertible matrix $P$ such that $B=PAP^{-1}$. Fillmore in \cite{Fillmore} shows that two matrices are similar if and only if both matrices have the same trace, using this result we enunciate the following result: 

\begin{corollary} 
Let $A,B$ be diagonal matrices in $R$ with entries $\alpha_1,\alpha_2,\cdots, \alpha_k \in F$ all different numbers. Then $A$ and $B$ are similar if and only if they have the same number of $\alpha_i's$ on their diagonals, for $i=1,2,\cdots,k$. \end{corollary}
\begin{proof}
Is well-know that, if $A$ and $B$ are similar matrices then they are the same number of eigenvalues, so, they are the same number of $\alpha_i's$.

Conversely, let $A=diag(a_1,a_2,\cdots,a_n)$ and $B=diag(b_1,b_2,\cdots,b_k)$, where $a_j,b_j \in\{\alpha_1,\alpha_2,\cdots,\alpha_k\}$, and denote by $r_i$ and $s_i$ the number of $\alpha_i$ on the diagonals of $A$ and $B$, respectively. Without loss of generality, we will consider that $a_1=\alpha_1$. Let $j_1$ the index such that $b_{j_1}=\alpha_1$ and $b_s\neq \alpha_1$ for all $s\in \{1,2,j_1-1\}$. Consider the transposition $(1,j_1)$ in the symmetric group $S_n$ and let $P$ the elementary matrix that permutes the rows $1$ and $j_1$ in $I_n$, the identity matrix in $M_n(F)$. Then, we have that $P=P^{-1}$ and $P^{-1}BP$ permutes $b_1$ with $b_{j_1}$, so, $P^{-1}BP=diag(\alpha_1,b_2,\cdots,b_{j_1-1},b_1,\cdots,b_n)$. Since $r_1=s_1$, repeat this process if necessary we obtain that $B$ is similar to $diag(\alpha_1,\cdots,\alpha_1,b_{r_1+1},\cdots,b_k)$. The same algorithm can be applied for $\alpha_2,\alpha_3,\cdots,\alpha_k$ and, from here we proof that $A$ and $B$ are similar.
\end{proof}

 Denote by $R=M_n(F)$ and by $U(R)$ the set of units in $R$ and by $U=|U(R)|$. Consider $\alpha_1,\alpha_2,\cdots ,\alpha_k \in F$ all different numbers and define the set 
 $$M(\alpha_1,\alpha_2,\cdots\!,\alpha_k)\!=\!\{A\in R,  if \ \gamma \ is \ an \ eigenvalue \ of \ A \ then \ \gamma\in \!\{  \alpha_1,\alpha_2, \cdots,\alpha_k\} \},$$
 and 
 $$E(\alpha_1,\alpha_2,\cdots, \alpha_k)=\{A \in M_n(F),  \alpha_1,\alpha_2,\cdots \  and \ \alpha_k \ are \ eigenvalues \ of \ A\}.$$
 Observe that $E(\alpha,\beta,\theta)\subset M(\alpha,\beta,\theta).$
 
 Now, we proof the main Theorem:
 
 \begin{proof}[Proof of Theorem \ref{th1}]
 Denote by
 $$N(s)=|\{(n_1,n_2,\cdots,n_s), \ such \ that \ n_1+n_2+\cdots+n_s=n, n_i\geq 0.\}|.$$
 Then the set $T=\{diag(a_1,a_2,\cdots,a_n), \ a_i \in \{\alpha_1,\alpha_2,\cdots, \alpha_k\} \}$ has $N(k)$ similarity classes. Denote by
 $$T_{n_1,n_2,\cdots,n_k}=diag(\underbrace{\alpha_1,\alpha_1,\cdots,\alpha_1}_{n_1 \ times},\underbrace{\alpha_2,\cdots,\alpha_2}_{n_2 \ times},\cdots,\underbrace{\alpha_k,\cdots,\alpha_k}_{n_k \ times}) \ \ \ (1)$$
 with $n_1+\cdots+n_k=n$, a set of representatives for these classes (observe here, that, if $n_i=0$ then $\alpha_i$ not appears in $T_{n_1,\cdots,n_{i-1},0,n_{i+1},\cdots,n_k}$). The equivalence classes are given by 
 $$[T_{n_1,\cdots,n_k}]=\{PT_{n_1,\cdots,n_k}P^{-1}, \ P\in U(R)\},$$
 where $n_1+\cdots+n_k=n$. Now we compute $|[T_{n_1,\cdots,n_k}]|$ to fixed  $(n_1,n_2,\cdots,n_k)$. Observe that $P_1T_{n_1,\cdots,n_k}P_1^{-1}=P_2T_{n_1,\cdots,n_k}P_2^{-1}$ iff $P_2^{-1}P_1\in C_{U(R)}(T_{n_1,\cdots,n_k})$, the centralizer of $T_{n_1,\cdots,n_k}$ in the group $U(R)$. Denote by $C_{n_1,\cdots,n_k}=|C_{U(R)}(T_{n_1,\cdots,n_k})|$ then we obtain
 $$|[T_{n_1,\cdots,n_k}]|=\frac{U}{C_{n_1,\cdots,n_k}}.$$
 Now, we compute $C_{n_1,\cdots,n_k}$. Let $T_{n_1,n_2,\cdots,n_k}$  be as (1)
 and assume that $A=(a_{ij})\in C_{U(R)}(T_{n_1,\cdots,n_k})$, then
 $$T_{n_1,\cdots,n_k}A=AT_{n_1,\cdots,n_k},$$
 and from here, we shown that $$A=\left[\begin{array}{cccc}A_1 & 0 & \cdots & 0 \\ 0 & A_2 & \cdots & 0 \\
 \vdots & &\ddots & \\ 0 & 0 & \cdots & A_k\end{array}\right],$$
 where $A_i\in U(M_{n_i}(F))$ for all $i=1,2,\cdots,k$. The converse of these affirmation is obviously true, then conclude that
 $$C_{n_1,\cdots,n_k}=\prod^k_{i=1}|U(M_{n_i}(F)|.$$
 Therefore, we have 
 $$|M(\alpha_1,\alpha_2,\cdots,\alpha_k)|= \sum_{\substack{n_1+n_2+\cdots+n_k=n \\ n_i \geq 0, \ i=1,2,\cdots, k}} \frac{U}{\prod^k_{i=1} U_{n_i}},$$
 The proof is similar to $|E(\alpha_1,\alpha_2,\cdots,\alpha_k)|$.
 \end{proof}
And, by the above Theorem we enunciate the following interesting result
\begin{corollary} Let $F$ be a finite field and $R=M_n(F)$ the $n\times n$ matrix group. Then  the number of idempotent matrices is equal to the number of involution matrices. Also, the number of $k$-potent matrices is equal to the number of matrices of finite order $k$.
\end{corollary}
\begin{proof}
Follows immediately from Theorem \ref{th1}.
\end{proof}
\begin{remark}
Furthermore, we proof that the quantity of idempotent matrices are equals to the number of all matrices with two different eigenvalues fixed (not necessarily, $0$ and $1$) and this is valid for any subset of specially eigenvalues (for example, eigenvalues of one $k$-potent matrix). The Theorem $\ref{th1}$ is a result that the authors suspected in the articles of Slowik \cite{Slowik}, Hou \cite{Hou} and Gargate \cite{Gargate3} but for full matrix group. 
\end{remark}
\section{$(k+1)$-potent elements over finite rings}
MacHale in \cite{Machale} found an upper bound for the number of idempotents of a finite ring $R$. Cheraghpour and Ghosseiri in \cite{Ghosseiri} found a smaller bound and they shown that 
$$|\mathcal{I}(R)|\leq \frac{2^n}{\prod^n_{i=1}p_i}|R|$$
where $p_1,p_2,\cdots,p_n$ are different primes dividing $|R|$ and $\mathcal{I}(R)$ denotes the set of all idempotent elements in $R$.
In this section we found an similar upper bound for the number of $(k+1)$-potent elements in $R$. Recall that an element $x\in R$ is called $(k+1)$-potent if and only if  $x^{k+1}=x$. First, we proof the following lemma:
\begin{lemma}\label{lema1}
Let $F$ be a finite field with $|F|=q$ and suppose that $\omega$ is an arbitrary $k$th root of unity in $F$, $n\geq 1$ and let $R=M_n(F)$ the $n\times n$ matrix group. Then, for $1 \leq k\leq n$, we have that
$$|\mathcal{I}_{k+1}(R)|\leq (k+1)q^{\displaystyle\frac{2n^2k}{k+1}-1}$$
where, $\mathcal{I}_{k+1}(R)$ denote the set of $(k+1)$-potent matrices.

\end{lemma}

\begin{proof} For $k=1$ we need compute all idempotent matrices and for this see \cite{Ghosseiri}. Suppose that $k>1$ and remember that, if $A$ is a $(k+1)$-potent matrix and if $\gamma$ is one eigenvalue then $\gamma \in \{0,1,\omega,\cdots,\omega^{k-1}\}$ where $\omega$ is a $k$th root of unity. So, to compute all $(k+1)$-potent matrices we observe that $\mathcal{I}_{k+1}(R)=M(0,1,\omega,\cdots,\omega^{k-1})$ and by the Theorem \ref{th1} we have that
$$\begin{array}{rl} |I_{k+1}(R)|=\!\!&\!\!\!\displaystyle\sum_{\substack{n_1+n_2+\cdots+n_{k+1}=n \\ n_i\geq 0, \ i=1,2,\cdots,k+1}} \frac{U_n}{\prod_{i=1}^{k+1}U_{n_i}}\\ 
=\!\!& \!\!\!\displaystyle \sum^{k+1}_{s=1}\sum_{\substack{n_1+\cdots+ n_{s}=n \\ n_i \geq 1,i=1,\!\cdots\!,s}}  \frac{q^{\binom{n}{2}}(q\!-\!1)\!\cdots\! (q^n-1)}{q^{\binom{n_1}{2}}(q\!-\!1)\!\cdots\!(q^{n_1}\!-\!1)\cdots q^{\binom{n_{k+1}}{2}}
(q-1)\!\cdots\!(q^{n_{k+1}}\!-\!1)}
\end{array}$$
For instance, denote by
$$H(s)=\displaystyle\sum_{\substack{n_1+\cdots+ n_{s}=n \\ n_i \geq 1, \ i=1,2,\cdots,s}}  \frac{q^{\binom{n}{2}}(q-1)\cdots (q^n-1)}{q^{\binom{n_1}{2}}(q-1)\cdots(q^{n_1}-1)\cdots q^{\binom{n_s}{2}}
(q-1)\cdots(q^{n_s}-1)}$$
then

$$\begin{array}{rl}\frac{q^{\binom{n}{2}}(q-1)\cdots (q^n-1)}{q^{\binom{n_1}{2}}(q-1)\cdots(q^{n_1}-1)\cdots q^{\binom{n_s}{2}}
(q-1)\cdots(q^{n_s}-1)}=\!&\!\!\!\!q^{\displaystyle\sum^{s}_{i,j=1,i< j}\!\!\!n_in_j}\!\!\!\cdot\frac{(q^{n_1+1}-1)\cdots(q^n\!-\!1)}{(q\!-\!1)\!\cdots \!(q^{n_2}\!-\!1)\!\cdots\! (q\!-\!1)\!\cdots \!(q^{n_s}\!-\!1)}\\
\leq\! &\!\!\!\! q^{\displaystyle\sum^{s}_{i,j=1,i< j}\!\!\!n_in_j}\!\!\!\cdot \displaystyle\frac{ q^{n_1+1}\cdots q^{n}}{q\cdots q^{n_2-1} \cdots q \cdots q^{n_s-1}} \\
=\!&\!\!\!\! \displaystyle q^{\displaystyle\sum^{s}_{i,j=1,i< j}\!\!\!\!\!\!\!n_in_j}\!\!\!\cdot\! \frac{q^{(n_2+\cdots+n_s)n_1+ \frac{(n-n_1)(n-n_1+1)}{2}}}{q^{\frac{(n_2-1)n_2}{2}} \cdots q^{\frac{(n_s-1)n_s}{2}}}\\
=& q^{\displaystyle\sum^{s}_{i,j=1,i< j}2n_in_j+ \sum^s_{i=2}n_i}
\end{array}$$
Now, we analyzing the function
$$F(x_1,x_2,\cdots,x_{s-1},x_s)=\displaystyle\sum^{s}_{i,j=1,i< j}2x_ix_j+ \sum^s_{i=2}x_i$$
with restriction $x_1+x_2+\cdots+x_s=n$.
The critical points are
$$x_1=\frac{2n-s+1}{2s} \ and \  x_i=\frac{2n+1}{2s}, i=2,\cdots,s.$$
Its not difficult proof that $$max(F)=2\left(\frac{2n+1}{2s}\right)^2\frac{(s-1)s}{2}.$$
So,
$$H(s)\leq \sum_{\substack{n_1+\cdots+n_s=n \\ n_i\geq 1, i=1,2,\cdots,s}}q^{\displaystyle \frac{(s-1)(2n+1)^2}{4s}}=N(s)q^{\displaystyle \frac{(s-1)(2n+1)^2}{4s}},$$
where $N(s)$ denote the number of integer solutions of the equation $n_1+\cdots+n_s=n$ with $n_i\geq 1$, $i=1,\cdots,s$.
Therefore
$$|\mathcal{I}_k(R)|=\sum_{s=1}^{k+1}H(s).$$
From here, we observe that
$$\begin{array}{rl}|\mathcal{I}_k(R)|\leq & \displaystyle \sum^{k+1}_{s=1}N(s)q^{\displaystyle \frac{(s-1)(2n+1)^2}{4s}}\\
\leq & \left[\displaystyle\sum^{k+1}_{s=1}N(s)\right]q^{\displaystyle\frac{k(2n+1)^2}{4(k+1)}}\end{array}$$
Observe that $N(1)=1, N(2)=n-1$ and $N(s)\leq (n-1)^s$, for $s=3,4,\cdots,k+1$. So,
$$\displaystyle \sum^{k+1}_{s=1}N(s)\leq (k+1)(n-1)^{k+1}\leq (k+1)^n$$
because $n^k\leq k^n$ for $k=2,\cdots,n$. Therefore
$$\begin{array}{rl}|\mathcal{I}_k(R)|\leq & (k+1)^{n}q^{\displaystyle \frac{k(2n+1)^2}{4(k+1)}} \\
\leq & (k+1)q^{n-1}q^{\displaystyle \frac{k(2n+1)^2}{4(k+1)}}\\
\leq & (k+1)q^{-1}q^{\displaystyle \frac{n^2k+(2k+1)n+1}{k+1}}\\
\leq  &(k+1)q^{\displaystyle\frac{2n^2k}{k+1}-1}\end{array}$$
and this conclude our proof.
\end{proof}

Follows immediately, from Lemma \ref{lema1} the same result of \cite{Ghosseiri} for $k=1$. Now, for $R$, a finite ring, we denote by $J=rad(R)$ the Jacobson radical of $R$. By the Wedderburn-Artin Theorem and Wedderburn's Little Theorem \cite{Lam}, we have 
$$\bar{R}=R/J\cong M_{n_1}(F_1)\times \cdots \times M_{n_s}(F_s),$$
where $n_i\geq 1$ and $F_i$ is a finite field, for all $i=1,2,\cdots,s$. Remember that for every pair of rings $A,B$ we have $U(A\times B)=U(A)\times U(B)$, then by the above Lemma we proof the following Theorem that generalize the result of MacHale's for $(k+1)$-potent elements in a ring.

\begin{proof}[Proof of Theorem \ref{th2}]
We have $R/J \cong M_{n_1}(F_1)\times \cdots \times M_{n_s}(F_s).$ Let $|R|=p^r$, $|J|=p^l$ and $|F_i|=p^{t_i}$, $i=1,2,\cdots, s$. Since for any $(k+1)-$potent $e \in R$, $\bar{e}$ is  $(k+1)-$potent in $R/J$ and by the Lemma \ref{lema1} we have
$$\begin{array}{rl}|\mathcal{I}_{k+1}(R)|=&|J|\displaystyle\prod^s_{i=1}|I_{k+1}(M_{n_i}(F_i))|\\
\leq & p^l\displaystyle\prod ^s_{i=1}(k+1) |F_i|^{\displaystyle \frac{2n_i^2k}{k+1}-1} \\
\leq & p^l(k+1)^s \displaystyle \prod^s_{i=1} p^{\displaystyle \frac{2n_i^2kt_i}{k+1}-t_i} \\
=& \displaystyle\frac{(k+1)^s p^{\displaystyle\frac{\sum^s_{i=1} 2n_i^2kt_i}{k+1}+l}}{p^{\displaystyle\sum^s_{i=1}t_i}}\end{array}$$
Hence,
$$\begin{array}{rl}
|\mathcal{I}_{k+1}(R)|
\leq & \frac{\displaystyle(k+1)^sp^{\displaystyle\frac{2k}{k+1}(\sum_{i=1}^sn_i^2kt_i+l)}}{\displaystyle p^{\displaystyle\sum^s_{i=1}t_i}} \\ &
\\
=& \frac{\displaystyle(k+1)^sp^{\displaystyle\frac{2k}{k+1}r}}{\displaystyle p^{\displaystyle\sum^s_{i=1}t_i}} \\
=& \displaystyle\frac{(k+1)^s|R|^{\displaystyle\frac{2k}{k+1}}}{{p^{\displaystyle\sum^s_{i=1}t_i}}} \\
\leq & \left(\displaystyle\frac{k+1}{p}\right)^s|R|^{\displaystyle\frac{2k}{k+1}} \\
\leq & \displaystyle\frac{k+1}{p} |R|^{\displaystyle \frac{2k}{k+1}}.
\end{array}$$

\end{proof}
Now we proof the Theorem \ref{th3} 
\begin{proof}[Proof of Theorem \ref{th3}] 
We use the follow remark in \cite{Ghosseiri}. For a finite ring $R$ there is a decomposition as direct sum of   rings of prime power order. This decomposition of $(R,+)$ is uniquely determined up to isomorphims, so we can write $R= R_1\oplus R_2\oplus \cdots \oplus R_s$, where $|R_i|=p_i^{r_i}$ and the primes $p_i$ are the distinct prime divisors of $|R|$. By the Theorem \ref{th2} and the fact that $|\mathcal{I}_{k+1}(R)|=\prod^{s}_{i=1}|\mathcal{I}_{k+1}(R_i)|   $ follows the result immediately.
\end{proof}
And, finally we  enunciate the follow corollary that the proof follows immediately.
\begin{corollary}
If $R$ is a finite ring and $p$ is the smallest prime dividing $|R|$, then $$|\mathcal{I}_{k+1}(R)|\leq \left(\frac{k+1}{p}\right)^n|R|^{\displaystyle\frac{2k}{k+1}},$$ where $\mathcal{I}_{k+1}(R)$ denote the set of $(k+1)$-potent elements in $R$ and $n$ is the number of distinct primes dividing $R$.
\end{corollary}

\section{Table}
The following table present the quantity  of matrices in $M_n(F)$  with $k$ different eigenvalues, where $F $ is a finite field and $|F|=q$.
\begin{center}\begin{tabular}{|c|c||l|} \hline
$n$ & $k$ & \hspace{3.53cm} $|E(\alpha_1,\alpha_2,\cdots,\alpha_k)|$\\ \hline \hline
$3$ & $2$ & $2q^4+2q^3+2q^2$ \\ \hline
$3$ & $3$ & $q^6+2q^5+2q^4+q^3$\\ \hline
$4$ & $2$ & $q^8+q^7+4q^6+3q^5+3q^4+2q^3$\\ \hline
$4$ & $3$ & $3q^{10}+6q^9+9q^8+9q^7+6q^6+3q^5$\\ \hline
$4$ & $4$ & $q^{12}+3q^{11}+5q^{10}+6q^9+5q^8+3q^7+q^6$\\ \hline
$5$ & $2$ & $2q^{12}+2q^{11}+4q^{10}+4q^9+6q^8+4q^7+4q^6+2q^5+2q^4$ \\ \hline
$5$ & $3$ & $3q^{16}+6q^{15}+15q^{14}+21q^{13}+27q^{12}+27q^{11}+24q^{10}+15q^{9}+$ \\
 & & $9q^8+3q^7$ \\ \hline
$5$ & $4$ & $4q^{18}+12q^{17}+24q^{16}+36q^{15}+44q^{14}+44q^{13}+36q^{12}+24q^{11}+$\\
& & $12q^{10}+4q^9$\\ \hline
$5$ & $5$ & $q^{20}+4q^{19}+9q^{18}+15q^{17}+20q^{16}+22q^{15}+20q^{14}+15q^{13}+$\\ 
& & $9q^{12}+4q^{11}+q^{10}$\\ \hline
$6$ & $2$ & $q^{18}+q^{17}+4q^{16}+5q^{15}+7q^{14}+7q^{13}+9q^{12}+6q^{11}+7q^{10}+5q^9+$\\
& &$4q^8+2q^7+2q^6+2q^5$ \\ \hline
$6$ & $3$ & $q^{24}+2q^{23}+11q^{22}+19q^{21}+35q^{20}+48q^{19}+65q^{18}+72q^{17}+74q^{16}+$\\
& & $67q^{15}+56q^{14}+41q^{13}+25q^{12}+15q^{11}+6q^{10}+3q^9$\\ \hline
$6$ & $4$ & $6q^{26}+18q^{25}+46q^{24}+84q^{23}+132q^{22}+178q^{21}+212q^{20}+224q^{19}+$\\
& & $210q^{18}+176q^{17}+128q^{16}+82q^{15}+42q^{14}+18q^{13}+4q^{12}$ \\ \hline
$6$ & $5$ & $5q^{28}+20q^{27}+50q^{26}+95q^{25}+150q^{24}+205q^{23}+245q^{22}+260q^{21}+$\\ 
& & $245q^{20}+205q^{19}+150q^{18}+95q^{17}+50q^{16}+20q^{15}+5q^{14}$ \\ \hline
$6$ & $6$ & $q^{30}+5q^{29}+14q^{28}+29q^{27}+49q^{26}+71q^{25}+90q^{24}+101q^{23}+101q^{22}+$\\
& &$90q^{21}+71q^{20}+49q^{19}+29q^{18}+14q^{17}+5q^{16}+q^{15}$\\ \hline

\end{tabular}
\end{center}

\end{document}